\documentclass[a4paper,12pt,reqno]{amsart}
\usepackage{pgf,tikz}
\usetikzlibrary[patterns]
\usepackage{mathrsfs}
\usetikzlibrary{arrows}

\usepackage{fancyhdr}
\usepackage{stmaryrd}
\usepackage[scale=0.75]{geometry}
\usepackage{mathtools}
\usepackage{color}
\usepackage{url}
\usepackage{amsmath}
\usepackage{amsthm}
\usepackage{amssymb}
\usepackage{amscd}          			%diagrams

\usepackage{bbm}            			%bold numbers
\usepackage{mathrsfs}           		%nice calligraphy

\usepackage[english]{babel}     		%if writing in French

\usepackage{graphicx}          		%uncomment for inclusion of graphic images

\usepackage{verbatim}           		%DO NOT comment this line

\theoremstyle{plain}
\newtheorem{Theorem}{Theorem}[section]
\newtheorem{Lemma}[Theorem]{Lemma}

\newtheorem{Proposition}[Theorem]{Proposition}
\newtheorem{Conjecture}[Theorem]{Conjecture}
\newtheorem{theorem}[Theorem]{Theorem}

\theoremstyle{definition}
\newtheorem{Definition}[Theorem]{Definition}
\newtheorem{Example}[Theorem]{Example}

\theoremstyle{remark}
\newtheorem{Remark}[Theorem]{Remark}
\newtheorem{Claim}{Claim}

\newcommand{\N}{\mathbb{N}}     %natural numbers
\newcommand{\R}{\mathbb{R}}     %real numbers
\newcommand{\Z}{\mathbb{Z}}         %integers

\newcommand{\calC}{\mathscr{C}}
\newcommand{\calD}{\mathscr{D}}
\newcommand{\calE}{\mathscr{E}}
\newcommand{\calF}{\mathscr{F}}
\newcommand{\calG}{\mathscr{G}}
\newcommand{\calH}{\mathscr{H}}

\newcommand{\calL}{\mathscr{L}}

\newcommand{\calQ}{\mathscr{Q}}

\DeclareMathOperator{\diam}{diam}			%diameter
\DeclareMathOperator{\dist}{dist}			%distance
\DeclareMathOperator{\diver}{div}			%divergence
\DeclareMathOperator{\Lip}{Lip}				%Lipschitz constant
\DeclareMathOperator{\supp}{supp}			%support
\usepackage[mathcal]{eucal}					%for the set of BV subsets of Rn

\newcommand{\Bk}{\color{black}}

\newcommand{\hel}
{
\hskip2.5pt{\vrule height8pt width.5pt depth0pt}
\hskip-.2pt\vbox{\hrule height.5pt width8pt depth0pt}
\,
}

\renewcommand{\geq}{\geqslant}
\renewcommand{\leq}{\leqslant}
\renewcommand{\epsilon}{\varepsilon}

\begin{document}

\title[Solvability of the divergence equation]{Solvability in weighted Lebesgue spaces\\of the divergence equation with measure data}

\author{Laurent Moonens}
\address{Universit\'e Paris-Saclay, CNRS, Laboratoire de math\'ematiques d'Orsay, 91405, Orsay, France}
\email{laurent.moonens@universite-paris-saclay.fr}

\author{Emmanuel Russ}
\address{Universit\'e Grenoble Alpes, Institut Fourier, 100 rue des maths, 38610, Gi\`eres, France}
\email{emmanuel.russ@univ-grenoble-alpes.fr}

\maketitle

\begin{abstract}
In the following paper, one studies, given a bounded, connected open set $\Omega\subseteq\R^n$, $\kappa>0$, a positive Radon measure $\mu_0$ in $\Omega$ and a (signed) Radon measure $\mu$ on $\Omega$ satisfying $\mu(\Omega)=0$ and $|\mu|\leq \kappa\mu_0$, the possibility of solving the equation $\diver u=\mu$ by a vector field $u$ satisfying $|u|\lesssim \kappa w$ on $\Omega$ (where $w$ is an integrable weight only related to the geometry of $\Omega$ and to $\mu_0$), together with a mild boundary condition.

This extends results obtained in \cite{DMRT} for the equation $\diver u=f$, improving them on two aspects: one works here with the divergence equation with \emph{measure} data, and also construct a weight $w$ that relies in a softer way on the geometry of $\Omega$, improving its behavior (and hence the \emph{a priori} behavior of the solution we construct) substantially in some instances.

The method used in this paper follows a constructive approach of Bogovskii type.
\end{abstract}

\section*{Introduction}

Let $\Omega\subset \R^n$ be an arbitrary bounded connected open subset. We are interested in the following question, stated here in a rather vague fashion: given a signed Radon measure $\mu$ on $\Omega$ with $\mu(\Omega)=0$, does there exist a vector field $u$ in a weighted $L^{\infty}$ space on $\Omega$ solving the boundary value problem:
\begin{equation}\label{eq.pbl}
\left\{
\begin{array}{ll}
\diver u=\mu \Bk & \mbox{ in }\Omega,\\
u\cdot \nu=0 \Bk & \mbox{ on }\partial\Omega,
\end{array}
\right.
\end{equation}
where the boundary condition for $u$ means at least that, for all $\varphi\in C^{\infty}(\R^n)$ having compact support in $\R^n$ (and hence allowed to be nonzero on and around the boundary of $\Omega$), one has:
$$
\int_{\Omega} u(x)\cdot \nabla \varphi(x)dx=-\int_{\Omega} \varphi(x)d\mu(x) ?
$$
\par
\noindent When $d\mu(x)=f(x)dx$, \Bk where $f\in L^{\infty}(\Omega)$ has zero integral, this question, which was widely studied when $f\in L^p(\Omega)$ with $1<p<+\infty$ and $\Omega$ is smooth (see, for instance, \cite[Section 7]{BB} for the construction of a solution in $W^{1,p}_0(\Omega)$), was previously addressed for arbitrary domains by Duran, Muschietti, Tchamitchian and the second author \cite{DMRT}. They characterized the bounded domains $\Omega$ with the following property: there exists an integrable weight $w>0$ in $\Omega$ such that, for all $f\in L^{\infty}(\Omega)$ with $\int_{\Omega} f(x)dx=0$, one can find a measurable vector field $u:\Omega\rightarrow \R^n$ such that $\left\vert u(x)\right\vert\lesssim w(x)$ for almost every $x\in \Omega$, solving
\begin{equation} \label{divermu}
\left\{
\begin{array}{ll}
\diver u=f \Bk & \mbox{ in }\Omega,\\
u\cdot \nu=0 \Bk & \mbox{ on }\partial\Omega,
\end{array}
\right.
\end{equation}
in the sense that:
\begin{equation}\label{eq.div-phi}
\int_{\Omega} u(x)\cdot \nabla\varphi(x)dx=-\int_{\Omega} f(x)\varphi(x)dx
\end{equation}
for all $\varphi\in L^1(\Omega)$ weakly differentiable and such that $w\nabla\varphi\in L^1(\Omega)$. Namely, they proved (see \cite[Theorem 2.1]{DMRT}) that the bounded domains $\Omega$ meeting this property are precisely the ones for which the geodesic distance to a fixed point in $\Omega$ is integrable in $\Omega$, and can also be described as the ones supporting a weighted $L^1$ Poincar\'e inequality.\Bk \par
We shall exhibit an example where the weight $w$ is unbounded in $\Omega$, but where it is still possible to construct a \emph{bounded} solution $u$ of \eqref{eq.pbl} where the boundary condition means at least that \eqref{eq.div-phi} holds for any test function $\varphi\in C^\infty(\R^n)$ with compact support in $\R^n$~---~hence allowed to be nonzero on, and around, $\partial\Omega$. In some situations, this can be achieved by working in an open measurable cover of $\Omega$, \emph{e.g.} in an open set $\hat{\Omega}\supseteq\Omega$ having the same Lebesgue measure as $\Omega$~---~see Example~\ref{ex.carre} below.\par
\noindent In the present paper, we improve the results of \cite{DMRT} replacing, in the right hand side of \eqref{divermu}, the function $f$ by a general signed Radon measure $\mu$ on $\Omega$ satisfying $\mu(\Omega)=0$~---~which, for some choices of $\mu_0$, forces the weight $w$ to be unbounded for some solution to exist in $L^\infty_{1/w}(\Omega,\R^n)$, even when $\Omega$ is smooth (see [Remark~\ref{rmk.init}, (ii)]). On the other hand, we formulate our results in an arbitrary open cover $\hat{\Omega}$ of $\Omega$ satisfying $|\hat{\Omega}|=|\Omega|$ (which can hence be $\Omega$ itself, or any open set containing $\Omega$ and contained in its essential interior for example), which, as we already explained, can, in some instances, yield a bounded weight (and hence a bounded solution to \eqref{divermu}) in $\Omega$ by choosing such a suitable cover.\\

Let us start by introducing more precisely the general framework of the paper, as well as stating more accurately its main results.

\section{Statements of the results} \label{results}
Throughout this paper, $n\geq 1$ is an integer and $m$ stands for the Lebesgue measure in $\R^n$. By ``domain'', we mean an open connected subset of $\R^n$. If $A,B$ are two nonempty subsets of $\R^n$, $d(A,B)$ denotes the distance between $A$ and $B$, that is $d(A,B)=\inf_{x\in A,\ y\in B} \left\vert x-y\right\vert$, where $\left\vert \cdot \right\vert$ is the Euclidean norm. If $E$ is a nonempty set and $A(f)$ and $B(f)$ are two nonnegative quantities for all $f\in E$, the notation $A(f)\lesssim B(f)$ means that there exists $C>0$ such that $A(f)\le CB(f)$ for all $f\in E$. Finally, for all open set $U\subset \R^n$, ${\mathcal D}(U)$ denotes the space of $C^{\infty}$ functions in $\R^n$ with compact support included in $U$.\par
\noindent Let us now state our precise results. Let $\Omega\subset \R^n$ be a bounded domain. {We choose, once and for all, a measurable cover $\tilde{\Omega}$ of $\Omega$, which in our case means that one has $\tilde{\Omega}\supseteq\Omega$ together with $m(\tilde{\Omega}\setminus\Omega)=0$ (see \cite[Definition 132D]{fremlin}). For all $x\in \tilde{\Omega}$, set:
$$
\tilde{d}(x):=d(x,\R^n\setminus \tilde{\Omega})
$$
and:
\begin{equation}\label{eq.omegatilde}
\hat{\Omega}:=\left\{x\in \tilde{\Omega}:\ \tilde{d}(x)>0\right\}.
\end{equation}
Note that $\hat{\Omega}$ is an open subset of $\R^n$.
\begin{Example}
For all $x\in \R^n$, say that $x$ is a point of density $1$ for $\Omega$ if
$$
\lim_{r\rightarrow 0} \frac{m(B(x,r)\cap\Omega)}{m(B(x,r))}=1.
$$
The {\it measure theoretic interior} of $\Omega$ is defined as the set ${\Omega_e}$ of points $x\in \R^n$ which have density $1$ for $\Omega$. Note that $\Omega\subset {\Omega_e}\subset \overline{\Omega}$, whereas the Lebesgue differentiation theorem shows that one has $m(\Omega_e\setminus\Omega)=0$. Hence $\Omega_e$ is a natural choice one can think of for the measurable cover $\tilde{\Omega}$, but all our results hold for a general cover of $\Omega$.
\end{Example}}
\begin{Example}
Let $\tilde{\Omega}$ be a fixed measurable cover of $\Omega$, and define $\hat{\Omega}$ according to (\ref{eq.omegatilde}). Since one has $\Omega\subseteq\hat{\Omega}\subseteq\tilde{\Omega}$, it is clear that $\hat{\Omega}$ is itself a measurable cover of $\Omega$. Moreover $\hat{\Omega}$ is obviously open, and since it verifies $\Omega\subseteq\hat{\Omega}\subseteq\bar{\Omega}$, it is straightforward to see that $\hat{\Omega}$ is also connected. We shall in the sequel define:
$$
\hat{d}(x):=d(x,\R^n\setminus \hat{\Omega})
$$
and let, for $\varepsilon>0$:
$$
\hat{\Omega}_{\varepsilon}:=\left\{x\in \hat{\Omega}:\ \hat{d}(x)>\varepsilon\right\}.
$$
From now on, many constructions will be done in the open set $\hat{\Omega}$.
\end{Example}
\par
%{Given a locally integrable function $f$ on $\Omega$, we say that $f$ has a Lebesgue point at $x$ if one has:
%$$
%\lim_{r\to 0} \frac{1}{m(B(x,r))} \int_{B(x,r)} |f(y)-f(x)|\,dy.
%$$
%It follows from the Lebesgue differentiation theorem (see \cite[Corollary~1, p.44]{EG}) that almost every point in $\Omega$ is a Lebesgue point for $f$.
%Moreover, given $f\in L^1_{\loc}(\Omega)$, we call the function $f^*$ defined on $\Omega$ by:
%$$
%f^*(x):=\begin{cases} \lim_{r\to 0} \frac{1}{m(B(x,r))} \int_{B(x,r)} f&\text{if this limit exists,}\\
%0&\text{otherwise,}\end{cases},
%$$
%the \emph{precise representative} of $f$. It follows that one has $f=0$ a.e.\ on $\Omega$ if and only if $f^*=0$.}
\par A \emph{curve} is a continuous map $\gamma:[a,b]\rightarrow \R^n$, where $a<b$ are real numbers. We will frequently identify $\gamma$ and $\gamma([a,b])$. Say that $\gamma$ is \emph{rectifiable} if there exists $M>0$ such that, for all $N\geq 1$ and all $a=t_0<...<t_N=b$, there holds $\sum_{i=0}^{N-1} \left\vert \gamma(t_{i+1})-\gamma(t_i)\right\vert\leq M$, \Bk and define the length of $\gamma$, $l(\gamma)$, as the supremum of $\sum_{i=0}^{N-1} \left\vert \gamma(t_{i+1})-\gamma(t_i)\right\vert$ \Bk over all possible choices of $N$ and $a=t_0<...<t_N=b$.\par

\noindent Let $x_0\in \Omega$ be a fixed point in $\Omega$. For all $x\in \hat{\Omega}$, define $d_{\hat{\Omega}}(x)$ as the infimum of the lengths of all rectifiable curves $\gamma$ joining $x$ to $x_0$ in $\hat{\Omega}$ (note that such a curve always exists since $\hat{\Omega}$ is open and rectifiably path-connected), and call $d_{\hat{\Omega}}$ the geodesic distance to $x_0$ in $\hat{\Omega}$.\par

\noindent Let now $\mu_0$ be a fixed nontrivial finite (positive) Radon measure in $\Omega$. We intend to solve $\diver u=\mu$, where $\mu$ belongs to the class of all finite signed Radon measures $\mu$ in $\Omega$ satisfying $\mu(\Omega)=0$ and $\left\vert \mu\right\vert\lesssim \mu_0$. It turns out that a solution of this problem involves the integrability of $d_{\hat{\Omega}}$ with respect to $\mu_0$. We therefore introduce the following condition, which \Bk may be satisfied or not, and does not depend on the choice of $x_0$\footnote{Indeed, if $y_0\in\Omega$ is another point in $\Omega$, and if $d'_{\hat{\Omega}}$ denotes the geodesic distance to $y_0$ in $\hat{\Omega}$, then we have $|d_{\hat{\Omega}}(x)-{d}'_{\hat{\Omega}}(x)|\leq \dist_{\hat{\Omega}}(x_0,y_0)$ for all $x\in\Omega$, where $\dist_{\hat{\Omega}}(x_0,y_0)$ denotes the geodesic distance in $\hat{\Omega}$ between $x_0$ and $y_0$. 
It hence follows that the integrability on $\Omega$ of $d_{\hat{\Omega}}$ and ${d}'_{\hat{\Omega}}$ are equivalent.}:
\begin{equation} \label{dOmega}
d_{\hat{\Omega}}\in L^1(\mu_0).
\end{equation}

Let us now state our first result:
\begin{theorem} \label{main}
Let $\Omega\subset \R^n$ be a bounded domain. Then, there exists $C>0$ such that, for all nontrivial finite (positive) Radon measures $\mu_0$ in $\Omega$ such that \eqref{dOmega} holds, one can find a measurable weight $w_0$ in $\Omega$ with the following properties:
\begin{enumerate}
\item[(A)] $w_0\in L^1(\Omega)$ and $w_0(x)>0$ for almost every\footnote{Here and after, when no explicit measure is specified, ``almost every'' and $L^p$ spaces are considered with respect to the Lebesgue measure.\Bk} $x\in \Omega$,
\item[(B)] for any finite, signed Radon measure $\mu$ in $\Omega$ satisfying $\mu(\Omega)=0$ and $|\mu|\leq  \kappa\mu_0$ for some real number $\kappa>0$, there exists a vector-valued function $u$ solving (\ref{eq.pbl})
and satisfying the following estimate:
    \begin{equation}%\label{estimLinfty}
    |u(x)|\leq  C\kappa  |w_0(x)|,
    \end{equation}
for a.e. $x\in\Omega$.\par
\noindent Here, by saying that $u$ solves (\ref{eq.pbl}) we mean that one has:
\begin{equation}\label{eq.g}
\int_{\Omega} u\cdot\nabla \varphi=-\int_\Omega \varphi\,d\mu,
\end{equation}
for any $\varphi\in \calD(\R^n)$, where $\calD(\R^n)$ stands for the set of all functions in $\calC^\infty(\R^n)$ whose support is a compact set.
%
%\item[(i)] $\mu_0$-a.e.\ $x\in\Omega$ is a Lebesgue point of $g$;
%\item[(ii)] for some $\epsilon>0$, $g^*$ is bounded on $V_\epsilon\cap\Omega$, where $$V_\epsilon:=\{x\in\R^n: \dist(x,\supp\mu_0)<\epsilon\}.$$
%\end{enumerate}}
\end{enumerate}
\end{theorem}
\begin{Remark}\label{rmk.init}
Let us formulate two first remarks.
\begin{itemize}
\item[(i)] The reader will notice, at this stage, that \eqref{eq.g} contains a weak Neumann-type boundary condition in the fact that test functions in $\calD(\R^n)$ are allowed to be nonzero on, and around, the boundary of $\Omega$; equation \eqref{eq.g} can hence be interpreted as an integration by parts formula where the boundary term is zero. We shall study, in section~\ref{sec.bdry}, how this condition can be strengthened by enlarging the set of test functions for which \eqref{eq.g} holds~---~showing in particular that \eqref{eq.g} extends to functions $\varphi\in\Lip(\R^n)$, where the latter notation stands for the space of all Lipschitz functions on $\R^n$.
\item[(ii)] In some cases, the weight $w_0$ constructed in the previous theorem \emph{must} be unbounded, even if $\Omega$ is a smooth domain, a ball for instance. Assume indeed, for example, that $\mu_0$ is a finite Radon measure in $\Omega$ satisfying (for some $0<\epsilon<n-1$) $\mu_0(B(a,r))\geq cr^{n-1-\epsilon}$ for all $0<r<r_0$ and some fixed $a\in\Omega$ satisfying $\mu_0(\{a\})=0$. 
Without loss of generality, we may assume that $0<m_0:=\mu_0(B(a,r_0))<\frac 12\mu_0(\Omega)$. Now construct a signed measure $\mu$ on $\Omega$ defined by:
$$
\mu:=\mu_0\hel B(a,r_0)-\frac{m_0}{\mu_0(\Omega)-m_0}\mu_0\hel (\Omega\setminus B(a,r_0)).
$$
It is clear that one has $\mu(\Omega)=0$, $|\mu|\leq \mu_0$ and $\mu(B(a,r))\geq cr^{n-1-\epsilon}$ for all $0<r<r_0$.
Now if one were able to solve (\ref{eq.pbl}) in the lines of the above theorem by $u\in L^\infty_{1/w_0}(\Omega,\R^n)$ with a bounded $w_0$, the Gauss-Green formula borrowed from \cite[Theorem~2.10]{PHUCTORRES} would imply
$$
cr^{n-1-\epsilon}\leq \mu(B(a,r))\leq \left\|\frac{u}{w_0}\right\|_\infty \int_{\partial B(a,r)} w_0\,d\calH^{n-1}\leq Cr^{n-1}
$$
for almost every $0<r<r_0$, which yields a contradiction.
\end{itemize}
\end{Remark}

Actually, the condition \eqref{dOmega} is also necessary for the existence of a weight $w_0$ meeting the conclusions of Theorem \ref{main}. These are both equivalent to an $L^1$ Poincar\'e inequality, as stated in the next theorem:
\begin{theorem} \label{equivpoincare}
Let $\Omega\subset \R^n$ be a bounded domain and $\mu_0$ be a nontrivial finite positive Radon measure, and define $d_{\hat{\Omega}}$ as before. The following conditions are equivalent:
\begin{enumerate}
\item[(a)] $d_{\hat{\Omega}}\in L^1(\mu_0)$,
\item[(b)] there exists a weight $w\in L^1(\Omega)$, $w>0$ a.e.\ satisfying the conclusions of Theorem~\ref{main},
\item[(c)] there exists an integrable weight $w\in L^1(\Omega)$, $w>0$ a.e.\ yielding the following Poincar\'e inequality for all locally Lipschitz functions $f$ on $\hat{\Omega}$ belonging to $L^1(\mu_0)$ whose local Lipschitz constant is bounded on $\hat{\Omega}$:
\begin{equation} %\label{poincare1} \tag{$P_1$}
\int_{\Omega} \left\vert f(x)-f_{\Omega}\right\vert d\mu_0\lesssim \int_{\Omega} \left\vert \nabla f\right\vert w,
\end{equation}
where $f_{\Omega}:=\frac 1{\mu_0(\Omega)}\int_{\Omega} fd\mu_0$.
\end{enumerate}
\end{theorem}

%\todo{
%Link with the Gauss-Green formula ? Il faut qu'on en reparle. Y a-t-il un lien avec les conditions du type $\mu(B(x,r))\lesssim r^{n-1}$ de Phuc Torres... \Bk
%}

Before starting, let us present the structure of the present paper by sketching how one can obtain its main Theorems.\\

The proof of Theorem~\ref{main} goes as follows. Mainly, one first constructs a solution $u$ to the equation $\diver u=\mu$ using a Bogovskii-type representation formula inspired by \cite{DMRT} (and relying on a previous work by Bogovskii \cite{bog}). The main idea is to represent $u$ as an integral of the form:
$$
u(x)=\int_\Omega G(x,y)\,d\mu(y),
$$
where $G(x,y)$ is a Bogovskii type kernel suitable for our problem, which satisfies growth estimates yielding the boundedness of $u$ with respect to some integrable weight. In this paper, we shall devote section~\ref{sec.bog} to the construction of a solution $u$ by means of such a representation formula, and to the study of the associated Bogovskii-type kernel. Let us just mention for now that the definition of this kernel heavily relies on a system of paths, borrowed from \cite{DMRT}, joining any point in $\Omega$ to a fixed one, in an almost ``geodesic'' fashion while remaining \emph{inside} the given measurable cover of $\Omega$ we work in (namely, $\hat{\Omega}$). It is then a combination of routine approximation arguments, and subtle properties of the paths system, that the vector field $u$ constructed using this approach, satisfies the boundary conditions implicitly contained in (\ref{eq.g}), and even stronger ones; we devote section~\ref{sec.bdry} to studying those boundary issues.\\

Let us mention that the equation \eqref{divermu}, with a measure valued right hand side, was widely studied in \cite{PHUCTORRES} in the whole space $\R^n$. To our best knowledge, the present work is the first time that a Bogovskii type approach is proposed to solve \eqref{divermu} in a bounded general domain.

As far as Theorem~\ref{equivpoincare} is concerned, the equivalence of the three stated properties will follow from duality arguments and, roughly speaking, from applying (some version of) Poincar\'e's inequality to the distance function $d_{\hat{\Omega}}$. Proving Theorem~\ref{equivpoincare} will be the purpose of section~\ref{Poinc}.\par

\noindent {\bf Acknowledgements:} 
This research was conducted during the visits of the authors to Laboratoire de Math\'ematiques d'Orsay and the Institut Fourier. The authors wish to thank the institutes for the kind hospitality. The second author was partially  supported by the French ANR project ``RAGE'' no.~ANR-18-BCE40-0012.

\section{A Bogovskii-type representation formula}\label{sec.bog}

Before we start describing the procedure announced in the introduction, we borrow  from \cite{DMRT} the construction of a system of paths in $\hat{\Omega}$, which relies on a decomposition of $\hat{\Omega}$ into Whitney cubes.

Recall that $\hat{d}$ denotes the distance function to the complement of $\hat{\Omega}$.
We fix as before $x_0\in\Omega$ and denote by $d_{\hat{\Omega}}$ the geodesic distance, in $\hat{\Omega}$, to $x_0$. Dilating the whole setting by some factor around $x_0$, we may moreover assume (which will be useful later for computational purposes) that one has:
\begin{equation} \label{dx015}
\overline{B(x_0,1)}\subseteq \Omega\quad\text{and}\quad \hat{d}(x_0)\geq 15.
\end{equation}
Applying the result in \cite[p.~800]{DMRT} to the open set $\hat{\Omega}$ and to $x_0$, we get a family of paths in $\hat{\Omega}$ which enjoys a series of properties. This family will be used in the next section to solve the divergence equation by a Bogovskii-type approach.
\begin{Lemma} \label{paths}
For all $y\in \hat{\Omega}$, there exists a rectifiable curve $\gamma_y:[0,1]\rightarrow \hat{\Omega}$ such that, writing $\gamma(t,y)=\gamma_y(t)$, the following properties hold:
\begin{enumerate}
\item[$(a)$] for all $y\in \hat{\Omega}$, $\gamma(0,y)=y$, $\gamma(1,y)=x_0$,
\item[$(b)$]
$(t,y)\mapsto \gamma(t,y)$ is measurable,
\item[$(c)$]
for all $x,y \in \Omega$ and all $r\leq \frac 12 \hat{d}(x)$,
    \begin{equation} \label{Ahlfors}
    l(\gamma_y \cap B(x, r)) \lesssim r
    \end{equation}
and
	\begin{equation} \label{length}
	l(\gamma_y) \lesssim d_{\hat{\Omega}}(y),
	\end{equation}
\item[$(d)$]
for all $\varepsilon >0$ small enough, there exists $\delta >0$ such that
$$
\forall y \in \hat{\Omega}_{\varepsilon}, \ \gamma_y \subset \hat{\Omega}_{\delta}.
$$
\end{enumerate}
\end{Lemma}

We now introduce once and for all the weight that will be used throughout the paper.
Let $\mu_0$ be a Radon measure on $\Omega$. Assume that the geodesic distance to $x_0$ in $\hat{\Omega}$, namely the function $d_{\hat{\Omega}}$, satisfies \eqref{dOmega}. Define a function $\omega$ on $\Omega$ by:
\begin{equation} \label{defomega}
\omega(x)=\mu_0\left(\left\{y\in \Omega:\ \mbox{there exists }t\in [0,1]\mbox{ such that }\left\vert \gamma(t,y)-x\right\vert\leq \frac 12 \hat{d}(x)\right\}\right).
\end{equation}
%where $\mbox{dist}_{\hat{\Omega}}(\gamma(y),x)$ is the lower bound of the geodesic distance in $\hat{\Omega}$ between $x$ and an arbitrary point in $\gamma(y)$.
\noindent We also define a localized\footnote{Note that, contrary to the usual definition, we integrate over $\Omega$ in the definition of $I_1\mu_0$.} version of the Riesz potential $\mu_0$, $I_1\mu_0$, by letting, for $x\in\Omega$
$$
I_1\mu_0(x):=\int_{\Omega} \left\vert x-y\right\vert^{-n+1}d\mu_0(y).
$$
%Since one can compute, using Fubini's theorem:
%$$
%\int_\Omega I_1\mu_0(x)\,dx=\int_\Omega \int_\Omega |x-y|^{1-n} \,dx\,d\mu_0(y)\leq \mu_0(\Omega) \int_{B(0,2\diam \Omega)} |z|^{1-n}\,dz<+\infty,
%$$
%we see that $I_1\mu_0$ is integrable (and hence finite almost everywhere) on $\Omega$, \emph{i.e.} that one has:
%\begin{equation}\label{eq.riesz}
%I_1\mu_0\in L^1(\Omega).
%\end{equation}
\noindent Define finally, for $x\in\Omega$:
\begin{equation}\label{wmu}
w_0(x):=I_1\mu_0(x)+\omega(x)\hat{d}(x)^{-n+1}.
\end{equation}
It is shown in \cite{DMRT} that, when $\mu_0=\calL^n$ is the Lebesgue measure in $\R^n$, one can actually work with the weight $w_0(x):=\omega(x)\hat{d}(x)^{1-n}$.\\

%and observe that the following two conditions hold:
%\begin{itemize}
%\item[(i)] $w_0(x)>0$ for all $x\in\Omega$ (recall that $I_1\mu_0$ never vanishes in $\Omega$);
%\item[(ii)] $w_0\in L^1(\Omega)$ (this follows from Lemma~\ref{omega} and from (\ref{eq.riesz})).
%\end{itemize}
%
%
\par The present section is devoted to proving Theorem~\ref{main} in the latter context by constructing a solution to the equation $\diver u=\mu$ in $L^\infty_{1/w_0}$ satisfying (\ref{eq.g}) for all $\varphi\in \calD(\R^n)$ (recalling that this includes some mild boundary condition on $u$)~---~we shall discuss in the next section how \eqref{eq.g} can, in some cases, be extended to a larger class of test functions (hence yielding a stronger boundary condition on $u$).

Let us restate Theorem~\ref{main} by making $w_0$ explicit.
\begin{Proposition} \label{solutioninfty}
%\todo{Modifier cet \'enonc\'e}
Assume that $\mu_0$ is a nontrivial finite (positive) Radon measure in $\Omega$ satisfying (\ref{dOmega}) and let $w_0$ be the weight defined by (\ref{wmu}). Then, the following properties hold:
\begin{enumerate}
\item[(A)] $w_0\in L^1(\Omega)$ and $w_0(x)>0$ for almost every $x\in \Omega$,
\item[(B)] for all $\kappa>0$, for any finite, signed Radon measure $\mu$ in $\Omega$ satisfying $\mu(\Omega)=0$ and $|\mu|\leq  \kappa\mu_0$ for some real number $\kappa>0$, there exists a vector-valued function $u$ satisfying the following two properties:
\begin{itemize}
\item[(i)]  for all $\varphi\in\calD(\R^n)$, one has:
\begin{equation}\label{eq.form-faible}
\int_\Omega u\cdot\nabla\varphi=-\int_\Omega \varphi\,d\mu,
\end{equation}
so that in particular $u$ solves weakly the equation $\diver u=\mu$ in $\Omega$;
\item[(ii)] for a.e.\ $x\in\Omega$, one has:
    \begin{equation}\label{estimLinfty}
    |u(x)|\leq  C\kappa  |w_0(x)|,
    \end{equation}
where $C>0$ only depends on $\Omega$ and the choice of the family $\gamma$.
\end{itemize}
\end{enumerate}
\end{Proposition}

Before to prove the above proposition, let us start by an example illustrating how working in a suitable measurable cover of $\Omega$ can change drastically the behavior of the weight $w_0$ constructed above.

\begin{figure}[h] \label{figurecarre}
\begin{center}
\definecolor{wqwqwq}{rgb}{0.3764705882352941,0.3764705882352941,0.3764705882352941}
\definecolor{qqqqff}{rgb}{0.,0.,1.}
\definecolor{ududff}{rgb}{0.30196078431372547,0.30196078431372547,1.}
\definecolor{ffqqqq}{rgb}{1.,0.,0.}
\definecolor{ttffqq}{rgb}{0.2,1.,0.}
\begin{tikzpicture}[line cap=round,line join=round,>=triangle 45,x=8.0cm,y=8.0cm]
\draw[->,color=black] (-0.1,0.) -- (1.1,0.);
\foreach \x in {,0.2,0.4,0.6,0.8,1.}
\draw[shift={(\x,0)},color=black] (0pt,2pt) -- (0pt,-2pt) node[below] {\footnotesize $\x$};
\draw[->,color=black] (0.,-0.1) -- (0.,1.1);
\foreach \y in {,0.2,0.4,0.6,0.8,1.}
\draw[shift={(0,\y)},color=black] (2pt,0pt) -- (-2pt,0pt) node[left] {\footnotesize $\y$};
\draw[color=black] (0pt,-10pt) node[right] {\footnotesize $0$};
\clip(-0.1,-0.1) rectangle (1.1,1.1);
\fill[line width=2.pt,color=ttffqq,fill=ttffqq,fill opacity=0.10000000149011612] (0.,0.) -- (1.,0.) -- (1.,1.) -- (0.,1.) -- cycle;
\fill[line width=1.2pt,dotted,color=wqwqwq,fill=wqwqwq,pattern=north east lines,pattern color=wqwqwq] (0.1,0.265) -- (0.9,0.265) -- (0.9,0.365) -- (0.1,0.365) -- cycle;
\draw [line width=2.pt,color=ffqqqq] (0.,0.42)-- (0.9,0.42);
\draw [line width=2.pt,color=ffqqqq] (0.1,0.21)-- (1.,0.21);
\draw [line width=2.pt,color=ffqqqq] (0.,0.1)-- (0.9,0.1);
\draw [line width=2.pt,color=ffqqqq] (0.,0.9)-- (0.9,0.9);
\draw [line width=2.pt,color=ffqqqq] (0.1,0.65)-- (1.,0.65);
\draw [line width=0.8pt,dash pattern=on 1pt off 1pt] (0.21152,0.9517)-- (0.95,0.95)-- (0.95,0.775)-- (0.05,0.775)-- (0.05,0.535)-- (0.95,0.535)-- (0.95,0.315)-- (0.05,0.315)-- (0.05,0.155)-- (0.95,0.155)-- (0.95,0.05);
\draw [line width=2.pt,color=qqqqff] (0.29,0.275)-- (0.29,0.315)-- (0.95,0.315)-- (0.95,0.535)-- node[below] {$\gamma_{(x,y)}$} (0.05,0.535)-- (0.05,0.775)-- (0.95,0.775)-- (0.95,0.95)-- (0.21152,0.9517) ;
\draw [line width=1.2pt,dotted,color=wqwqwq] (0.1,0.265)-- (0.9,0.265);
\draw [line width=1.2pt,dotted,color=wqwqwq] (0.9,0.265)-- (0.9,0.365);
\draw [line width=1.2pt,dotted,color=wqwqwq] (0.9,0.365)-- (0.1,0.365);
\draw [line width=1.2pt,dotted,color=wqwqwq] (0.1,0.365)-- (0.1,0.265);
\begin{scriptsize}
\draw[color=ffqqqq] (0.29949935779458353,0.45351047575946485) node {$L_{k-1}$};
\draw[color=ffqqqq] (0.7627437044897136,0.2278273324977353) node {$L_k$};
\draw[color=ffqqqq] (0.2971237457602495,0.054407653991353666) node {$L_{k+1}$};
\draw[color=ffqqqq] (0.6629679990476856,0.859740133630578) node {$L_0$};
\draw[color=ffqqqq] (0.7603680924553796,0.6768180069868603) node {$L_1$};
\draw [fill=ududff] (0.21152,0.9517) circle (1.5pt);
\draw[color=ududff] (0.11182600708219753,0.9500133909352697) node {$(x_0,y_0)$};
\draw [fill=black] (0.95,0.155) circle (0.5pt);
\draw[color=black] (0.9563560852879345,0.027088115596512724) node {$\Gamma$};
\draw [fill=ududff] (0.29,0.275) circle (1.5pt);
\draw[color=ududff] (0.3315701202580925,0.2290151385149023) node {$(x,y)$};
\draw [fill=ududff] (0.29,0.315) circle (1.0pt);
\draw[color=wqwqwq] (0.5465630093653195,0.384617726763779) node {$Q_k$};
\end{scriptsize}
\end{tikzpicture}
\caption{The set $\Omega:=(0,1)^2\setminus\bigcup_{k\in\N} L_k$ from Example~\ref{ex.carre}}
\end{center}
\end{figure}
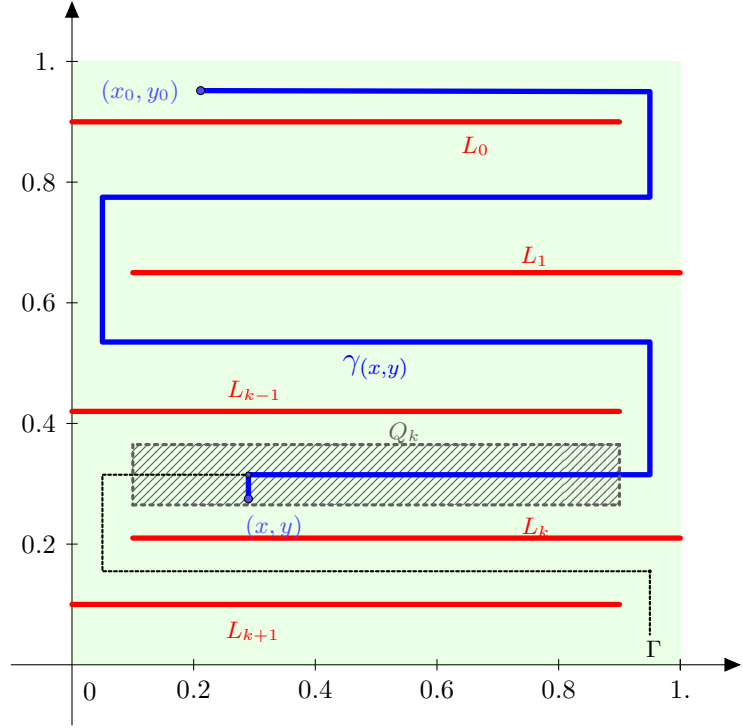

\begin{Example}\label{ex.carre}
Pick up a sequence $(h_k)_{k\in\N}\subseteq (0,1)$ strictly decreasing to $0$, let $\epsilon>0$ be small, and let for $k\in\N$:
$$
L_k:=\begin{cases} [0,1-\epsilon]\times\{h_k\}&\text{ if }k\text{ is even,}\\
[\epsilon,1]\times\{h_k\}&\text{ if }k\text{ is odd.}\end{cases}
$$
Define $\Omega:=(0,1)^2\setminus\bigcup_{k\in\N} L_k$. For all $(x,y)\in \Omega$, consider the path $\gamma_{(x,y)}$ drawn in the above picture and observe that this family of paths satisfies all conditions stated in Lemma~\ref{paths}. For this example, take for $\mu_0$ the (two-dimensional) Lebesgue measure. Fix $(x_0,y_0)$ as on the picture (with $y_0>h_0$ and, say, $0<x_0$ small), denote by $d_\Omega(x,y)$ the geodesic distance from $(x,y)$ to $(x_0,y_0)$ in $\Omega$, and let $d(x,y)$ be the distance from $(x,y)$ to the boundary of $\Omega$. \par
\noindent If $(x,y)\in \Omega$, there exists $k\in \N$ such that $h_{k+1}<y\leq h_k$. It is plain to see that 
$$
d_{\Omega}(x,y)\lesssim k(1-\epsilon).
$$
It follows that
\begin{eqnarray*}
\int_{\Omega} d_{\Omega}(x,y)dxdy & = & \sum_{k\in \N} \int_{h_{k+1}<y\leq h_k} d_{\Omega}(x,y)dxdy \\
& \lesssim & (1-\epsilon) \sum_{k\in \N} k(h_k-h_{k+1}),
\end{eqnarray*}
which entails that $d_{\Omega}\in L^1(\Omega)$ (for the Lebesgue measure) provided that 
\begin{equation} \label{cond1}
\sum_{k\in \N} k(h_k-h_{k+1})<+\infty.
\end{equation}
For all $k$, denote by $Q_k$ the set of points $(x,y)\in \Omega$ such that $\epsilon<x<1-\epsilon$ and 
$$
\left\vert y-\frac{h_k+h_{k+1}}2\right\vert\leq \eta (h_k-h_{k+1})
$$
for $\eta>0$ small enough (and independent from $k$). It is clear that, for all $(x,y)\in Q_k$, one has:
$$
d(x,y)\leq h_k-h_{k+1}.
$$
Moreover, for all $(\tilde{x},\tilde{y})\in \Omega$ such that $\epsilon<\tilde{x}<1-\epsilon$ and $\tilde{y}\leq h_{k}$, the path $\gamma_{(\tilde{x},\tilde{y})}$ intersects the ball $B\left((x,y),\frac 12 d(x,y)\right)$. It follows that:
$$
\omega(x,y)\geq h_k(1-2\epsilon).
$$
As a consequence,
\begin{equation} \label{cond2}
w(x,y)\gtrsim \frac{h_k(1-2\epsilon)}{h_k-h_{k+1}}=\frac{1-2\epsilon}{1-\frac{h_{k+1}}{h_k}}.
\end{equation}
Choosing, for instance, $h_k:=\frac 1{(k+1)^3}$, it is obvious that \eqref{cond1} holds and \eqref{cond2} shows that $w$ is unbounded in $\Omega$. \par
\noindent If, instead of $\Omega$, we consider now $\tilde{\Omega}:=(0,1)^2$, which is obviously a measurable cover of $\Omega$, the associated weight, denoted $\tilde{w}$, satisfies $\tilde{d}(x,y)\lesssim\tilde{w}(x,y)\lesssim 1+ \tilde{d}(x,y)$, and is therefore bounded in $\Omega$, where $\tilde{d}$ denotes the distance to the boundary of $\tilde{\Omega}$.
\end{Example}

The proof of Proposition \ref{solutioninfty} relies on several lemmata. The first one is an easy observation about $w_0$.
\begin{Lemma} \label{propI1mu0}
Let $w_0$ be defined as before.
\begin{enumerate}
\item For every $x\in \Omega$, $I_1\mu_0(x)>0$.
\item For all $p\in [1,\frac{n}{n-1})$, $I_1\mu_0\in L^p(\Omega)$. In particular, $I_1\mu_0(x)<+\infty$ for almost every $x\in \Omega$.
\end{enumerate}
\end{Lemma}
\begin{proof} [Proof of Lemma \ref{propI1mu0}] That one has $I_1\mu_0(x)>0$ for every $x\in\Omega$ follows at once from its definition and the fact that $\mu_0$ is non trivial. Let now $p\in [1,\frac{n}{n-1})$, $p^{\prime}$ defined by $\frac 1p+\frac 1{p^{\prime}}=1$ and $g\in L^{p^{\prime}}(\Omega)$ with $\left\Vert g\right\Vert_{p^{\prime}}=1$. Using H\"older's inequality and Fubini's theorem, one obtains:
\begin{multline*}
\left|\int_\Omega I_1\mu_0(x)g(x)\,dx\right|=\left|\int_\Omega \left(\int_\Omega |x-y|^{1-n} g(x)\,dx\right)\,d\mu_0(y)\right|\\ \lesssim \left(\int_{B(0,2\diam\Omega)} \left\vert z\right\vert^{p(1-n)}dx\right)^{\frac 1p}\mu_0(\Omega),
\end{multline*} 
which ends the proof. \end{proof}

\noindent A second preparatory lemma provides the integrability of $\omega$ against some power of the distance function to $\R^n\setminus\hat{\Omega}$.
\begin{Lemma} \label{omega}
 \begin{equation}\label{intomega}
    \int_\Omega \omega(x) \hat{d}(x)^{-n+1} dx < + \infty.
    \end{equation}
\end{Lemma}

\begin{proof}
We follow the proof of \cite[Lemma 2.3]{DMRT}, indicating only the main differences. Let $x\in \Omega$. If $\hat{d}(x)\geq \frac{60}7$, then one has $\omega(x)\hat{d}(x)^{-n+1}\leq C<+\infty$. We can therefore assume that $\hat{d}(x)<\frac{60}7$. In this case, let $y\in \Omega$ be such that there exists $t_0$ satisfying $\left\vert \gamma(t_0,y)-x\right\vert \leq \frac 12 \hat{d}(x)$.

Since we also have $15-\left\vert x-x_0\right\vert\leq \hat{d}(x_0)-\left\vert x-x_0\right\vert\leq \hat{d}(x)$ (recall that \eqref{dx015} holds), we obtain:
\begin{eqnarray*}
\left\vert \gamma(t_0,y)-x_0\right\vert & \geq & \left\vert x-x_0\right\vert - \left\vert x-\gamma(t_0,y)\right\vert \\
& \geq & \left\vert x-x_0\right\vert -\frac 12\hat{d}(x)\\
& \geq & 15-\frac 32 \hat{d}(x)\\
& \geq & \frac 74 \hat{d}(x)-\frac 32 \hat{d}(x)=\frac 14 \hat{d}(x).
\end{eqnarray*}
This implies that $x_0\notin B(\gamma(t_0,y),\frac 14 \hat{d}(x))$ and hence that one has: $$\frac 14 \hat{d}(x)\leq l(\gamma_y\cap B(\gamma(t_0,y),\frac 14 \hat{d}(x)))\ ;$$ the rest of the proof is then virtually identical to that of \cite[Lemma~2.3]{DMRT}.
\end{proof}

\medskip

\begin{proof}[Proof of Proposition~\ref{solutioninfty}] Observe that (A) is an immediate consequence of Lemmas~\ref{propI1mu0} and \ref{omega}.

We give a constructive proof of (B) relying on ideas going back to Bogovskii \cite{bog}, which was extended to John domains in \cite{adm} and was further generalized to arbitrary domains in \cite{DMRT}. We here adapt to our context arguments from the proof of \cite[Lemma~2.4]{DMRT}, completing them at some specific points.

\par

\medskip 
Let from now on $B_0:=B(x_0,1)$; it is clear by (\ref{dx015}) that one has $\overline{B_0}\subseteq\Omega$. We choose a function $\chi\in {\mathcal D}(\Omega)$ supported in $\overline{B_0}$ and such that $\int_{\Omega} \chi(x)dx=1$. For each $y\in \Omega\setminus\{x_0\}$, let $\tau(y)$ be the smallest $t>0$ such that $\gamma(t,y) \in \partial B(y, \frac12 \hat{d}(y))$ in case there exists a $t\in [0,1]$ for which one has $\gamma(t,y)\in \partial B(y,\frac 12 \hat{d}(y))$~---~call this ``case 1''~---~, and let $\tau(y) = 1$ otherwise~---~call this ``case 2''. We define a function $t\mapsto \rho(t,y)$, $t\in [0,1]$, by letting, in case 1:
$$
\begin{array}{ll}
\rho(t,y)=\alpha \left\vert y - \gamma(t,y) \right\vert &\mbox{ if }t\leq \tau(y),\\
\rho(t,y)=\frac 1{\hat{d}(x_0)} \hat{d}(\gamma(t,y)) &\mbox{ if }t>\tau(y),
\end{array}
$$
where $\alpha$ is so chosen that $\rho(\cdot, y)$ is a continuous function~---~this means that we have to take
$$
\alpha = \frac{2}{\hat{d}(x_0)} \frac{\hat{d}(\gamma(\tau(y),y))}{\hat{d}(y)}.
$$
In case 2, we let, for $0\leq t\leq 1$:
$$
\rho(t,y):= t.
$$

\begin{Claim}\label{claim.gamma} For all $t\in [0,1]$ and all $z\in B_0$, we have $\gamma(t,y)+\rho(t,y)(z-x_0) \in \hat{\Omega}$.\end{Claim}

To prove this claim, it is enough to check that one has:
\begin{equation}\label{a}
\rho(t,y)\leq \frac 1{5}\hat{d}(\gamma(t,y)).
\end{equation}
Observe that in case $1$, for $0\leq t \leq \tau(y)$, we have $\left\vert y - \gamma(t,y) \right\vert \leq \frac12 \hat{d}(y)$, which implies $\rho(t,y) \leq \frac{\alpha}{2} \hat{d}(y)$ and, in turn, $\rho(t,y)\leq\alpha \hat{d}(\gamma(t,y))$, for it is clear that we have:
$$
\hat{d}(y)\leq |y-\gamma(t,y)|+\hat{d}(\gamma(t,y))\leq \frac 12 \hat{d}(y)+\hat{d}(\gamma(t,y)),
$$
and hence also $\hat{d}(y)\leq 2 \hat{d}(\gamma(t,y))$.
We then have $\hat{d}(\gamma(\tau(y),y))\leq \hat{d}(y)+\left\vert \gamma(\tau(y),y)-y\right\vert=\frac 32 \hat{d}(y)$, and hence also $\alpha \leq \frac 15$.
By construction this finally yields \eqref{a}, and therefore
Claim~\ref{claim.gamma} in case 1.

In case 2, it is clear that we have $x_0\in B(y,\frac 12 \hat{d}(y))$; in particular this yields $|y-x_0|\leq\frac 12 \hat{d}(y)$ and we hence have:
$$
\hat{d}(y)\geq \hat{d}(x_0)-|y-x_0|\geq 15-\frac 12 \hat{d}(y),
$$
which implies $\frac 32 \hat{d}(y)\geq 15$. But for $0\leq t\leq 1$ we also have $\gamma(t,y)\in B(y,\frac 12 \hat{d}(y))$ so that:
$$
\hat{d}(\gamma(t,y))\geq  \hat{d}(y)-|y-\gamma(t,y)|\geq \hat{d}(y)-\frac 12 \hat{d}(y)=\frac 12 \hat{d}(y)\geq 5 \geq 5\rho(t,y).
$$
%This yields, for all $t\in (0,1)$, \Bk $\rho(t,y)\leq\frac 15 \hat{d}(\gamma(t,y))$ as well as:
%$$
%\gamma(t,y)+\rho(t,y)(z-x_0) \in \hat{\Omega}.
%$$
This completes the proof of Claim~\ref{claim.gamma}.\\

Fix now $\varphi \in \calD(\R^n)$. Using the fact that $m(\hat{\Omega}\setminus \Omega)=0$ and proceeding exactly as in the proof of \cite[Lemma~2.4]{DMRT}, we compute:
\begin{equation}\label{muphi}
\int_\Omega\varphi\,d\mu=-\int_\Omega\int_\Omega G(x,y)\cdot\nabla\varphi(x)\,dx\,d\mu(y),
\end{equation}
where $G(x,y)$ is defined as follows for $x,y\in \Omega$, $x\neq y$:
\begin{equation}\label{GG}
G(x,y) := \int_0^1 \left[\dot{\gamma}(t,y) + \dot{\rho}(t,y)\frac{x-\gamma(t,y)}{\rho(t,y)}\right] \chi\left(x_0+\frac{x-\gamma(t,y)}{\rho(t,y)}\right) \frac{dt}{\rho(t,y)^n}.
\end{equation}

As the following lemma shows, $G(x,y)$ is well-defined for a.e.\ $x\in\Omega$ and $\mu$-a.e.\ $y\in\Omega$.
\begin{Lemma} \label{Gxy}
For a.e.\ $x\in\Omega$ and $\mu$-a.e.\ $y\in\Omega$, $G(x,y)$ is well defined and one has:
$$
\int_{\Omega} \left\vert G(x,y)\right\vert d\left\vert \mu\right\vert (y)\lesssim \kappa w_0(x).
$$
\end{Lemma}
\begin{proof}
Note first that the integrand in $G(x,y)$ vanishes unless one has $|x-\gamma(t,y)|<\rho(t,y)$).
\begin{comment}
, $y \in \Omega$ and $z \in B_0$. Since the map $t\mapsto \varphi[\gamma(t,y)+\rho(t,y)(z-x_0)]$ is Lipschitz, the fundamental theorem of calculus applies and we have:
$$
\varphi (y) - \varphi (z) = - \int_0^1 [\dot{\gamma}(t,y) + \dot{\rho}(t,y)(z-x_0)] \cdot \nabla\varphi[\gamma(t,y) + \rho(t,y)(z-x_0)] dt,
$$
where we denoted by a dot the derivative with respect to the variable $t$.

Multiplying the latter equality by $\chi (z)$ and integrating, we get:
\begin{multline*}
\varphi (y) - \int_{\Omega}\varphi\chi = - \int_{B_0}\int_0^1\{ [\dot{\gamma}(t,y) + \dot{\rho}(t,y)(z-x_0)]\\ \cdot \nabla\varphi[\gamma(t,y) + \rho(t,y)(z-x_0)] \} \chi(z) dt\,dz.
\end{multline*}
Changing $z$ into $x = \gamma(t,y) + \rho(t,y)(z-x_0)$ and noticing that we have $B(\gamma(t,y),\rho(t,y))\subset \hat{\Omega}$ and that $m(\hat{\Omega}\setminus \Omega)=0$, this formula becomes:
\begin{multline*}
\varphi (y) - \int_{\Omega}\varphi\chi = - \int_{\Omega}\int_0^1 \left[\dot{\gamma}(t,y) + \dot{\rho}(t,y)\frac{x-\gamma(t,y)}{\rho(t,y)}\right] \\\cdot \nabla\varphi(x) \chi\left[x_0+\frac{x-\gamma(t,y)}{\rho(t,y)}\right] \frac1{\rho(t,y)^n} dt\,dx.
\end{multline*}
Since $\mu(\Omega)= 0$, this implies
\begin{multline}\label{muphi}
\int_{\Omega}\varphi \,d\mu = - \int_{\Omega}\int_{\Omega}\int_0^1 \left[\dot{\gamma}(t,y) + \dot{\rho}(t,y)\frac{x-\gamma(t,y)}{\rho(t,y)}\right] \\\cdot \nabla\varphi(x) \chi\left[x_0+\frac{x-\gamma(t,y)}{\rho(t,y)}\right] \frac1{\rho(t,y)^n} dt\,dx\,d\mu(y).
\end{multline}
\end{comment}
We now let $\Omega':=\{y\in\Omega:\tau(y)<1\}$ and, as in the proof of \cite[Lemma~2.5]{DMRT}, we write for $y\in\Omega'$:
$$
G(x,y) = G_1 (x,y) + G_2 (x,y)
$$
with
\begin{multline*}
G_1 (x,y) = \int_0^{\tau(y)} \left[\dot{\gamma}(t,y) - \left(\frac{\dot{\gamma}(t,y) \cdot (y-\gamma(t,y))}{ \left\vert y - \gamma(t,y) \right\vert^2}\right)(x-\gamma(t,y))\right] \\\cdot \chi\left(x_0+\frac{x-\gamma(t,y)}{\alpha \left\vert y - \gamma(t,y) \right\vert}\right) \frac{1}{\alpha^n  \left\vert y - \gamma(t,y) \right\vert^n} dt,
\end{multline*}
and
\begin{multline}\label{G2}
G_2 (x,y) = \int_{\tau(y)}^1 \left[\dot{\gamma}(t,y) + [\dot{\gamma}(t,y) \cdot \nabla \hat{d}(\gamma(t,y))] \frac{x-\gamma(t,y)}{\hat{d}(\gamma(t,y))} \right] \\\cdot\chi\left(x_0+\hat{d}(x_0)\frac{x-\gamma(t,y)}{\hat{d}(\gamma(t,y))}\right) \frac{[\hat{d}(x_0)]^n}{\hat{d}(\gamma(t,y))^n} dt.
\end{multline}

Proceeding as in the proof of \cite[Lemma~2.5]{DMRT}, we get the following estimate for $x,y\in\Omega$:
$$
\left\vert G_1 (x,y) \right\vert \leq C \left\vert x - y \right\vert^{-n+1}.
$$
We hence have that:
\begin{equation} \label{estimG1}
\int_{\Omega'} \left\vert G_1 (x,y) \right\vert d\left\vert\mu\right\vert(y) \leq C \kappa I_1\mu_0(x).
\end{equation}

Following once more the proof of \cite[Lemma~2.5]{DMRT}, we also show that one has:
\begin{equation} \label{estimG2}
\int_{\Omega'} \left\vert G_2(x,y)\right\vert d\left\vert\mu\right\vert(y)\leq C\kappa \omega(x)\hat{d}(x)^{-n+1},
\end{equation}
and gathering \eqref{estimG1} and \eqref{estimG2} yields:
\begin{equation}\label{G}
\int_{\Omega'} |G(x,y)|\,d\left\vert \mu\right\vert (y)\leq C\kappa w_0(x),
\end{equation}
for a.e. $x\in\Omega$.

If we fix now \Bk $y\in\Omega\setminus\Omega'$ (meaning that we are in case 2), we compute $\gamma(t,y)=y+t(x_0-y)$, $\gamma(t,y)+\rho(t,y)(z-x_0)=y+t(z-y)$, $\dot{\gamma}(t,y)=x_0-y$ and $\dot{\rho}(t,y)=1$. We hence get, from (\ref{GG}):
$$
G(x,y)=\int_0^1 \frac{x-y}{t}\cdot \chi\left(y+\frac{x-y}{t}\right)\,\frac{dt}{t^n}.
$$
Yet in order for the integrand in the above integral to be nonzero, we should have $y+\frac{x-y}{t}\in B_0$, implying in particular that one has:
$$
\left|y+\frac{x-y}{t}-x_0\right|<1.
$$
We hence compute:
$$
\left|\frac{x-y}{t}\right|\leq 1+ |x_0-y|\leq 1+\diam \Omega.
$$
Letting $c:=(1+ \diam \Omega)^{-1}$, we get in particular $t\geq c|x-y|$ and hence also:
$$
|G(x,y)|\leq \int_{c|x-y|}^1 \frac{dt}{t^n}\leq \frac{c^{1-n}}{1-n} |x-y|^{-n+1}\leq C |x-y|^{-n+1}.
$$
Integrating over $\Omega\setminus\Omega'$, we get:
\begin{equation} \label{Gbis}
\int_{\Omega\setminus\Omega'} |G(x,y)|\,d |\mu|(y)\leq C\kappa I_1\mu_0(x)\leq C \kappa w_0(x).
\end{equation}
According to \eqref{G} and \eqref{Gbis}, \Bk we have shown that:
\begin{equation}\label{eq.19}
\int_{\Omega} |G(x,y)|\,d|\mu|(y)=\int_{\Omega'} |G(x,y)|\,d|\mu|(y)+\int_{\Omega\setminus\Omega'} |G(x,y)|\,d|\mu|(y)\leq C\kappa w_0(x),
\end{equation}
which concludes the proof of Lemma \ref{Gxy}.
\end{proof}

As in \cite{DMRT}, we define $u$ by
$$
u(x) = \int_\Omega G(x,y) \, d\mu(y),
$$
which is well-defined by Lemma \ref{Gxy}, and we have, for a.e. $x\in\Omega$:
\begin{equation}\label{u}
\left\vert u(x) \right\vert \leq C\kappa w_0(x),
\end{equation}
which is exactly (\ref{estimLinfty}). It then follows from (\ref{muphi}) and Fubini's theorem, which we may apply thanks to Lemmata \ref{omega} and \ref{Gxy}, that we have, for any $\varphi\in\calD(\R^n)$:
\begin{equation}\label{eqinDOmega}
\int_{\Omega} u\cdot\nabla \varphi = - \int_{\Omega} \varphi \, d\mu,
\end{equation}
which is (\ref{eq.form-faible}).
\end{proof}
\begin{Remark}

In the context of the preceding proof, assume moreover that, for some $\epsilon>0$, $\mu_0$ satisfies $\mu_0(B(x,r))\lesssim r^{n-1+\epsilon}$ for all $x\in\Omega$ and all $0<r<\hat{d}(x)$. Then one can compute, for $x\in\Omega$ (calling $C>0$ a constant such that $G_1(x,y)$ vanishes unless one has $|x-y|\leq C\hat{d}(x)$, see \cite[p. 804]{DMRT}):
$$
\int_{|y-x|\leq C \hat{d}(x)} \frac{1}{|x-y|^{n-1}}\, d\mu_0(y)\leq\sum_{k=0}^\infty \int_{2^{-k-1}C\hat{d}(x)< |y-x|\leq 2^{-k}C\hat{d}(x)}  \frac{1}{|x-y|^{n-1}}\, d\mu_0(y).
$$
Yet we have for $k\in\N$:
\begin{multline*}
\int_{2^{-k-1}C\hat{d}(x)< |y-x|\leq 2^{-k}C\hat{d}(x)}  \frac{1}{|x-y|^{n-1}}\, d\mu_0(y)\leq (2^{-k-1}C\hat{d}(x))^{1-n}\mu_0[B(x,2^{-k}C\hat{d}(x)]\\
\lesssim 2^{n-1} C^\epsilon 2^{-k\epsilon} [\hat{d}(x)]^\epsilon.
\end{multline*}
It hence follows that one has:
$$
\int_\Omega G_1(x,y)\,d\mu_0(y)\lesssim C_\epsilon [\hat{d}(x)]^\epsilon,
$$
and that one could hence prove Proposition~\ref{solutioninfty} with a weight of the form $w_0(x)=\omega(x)[\hat{d}(x)]^{1-n}+C_\epsilon [\hat{d}(x)]^\epsilon$.
\end{Remark}

We now examine how (\ref{eqinDOmega}) can, in some cases, be extended to a wider class of test functions~---~hence extending, in some sense, the mild ``boundary condition'' appearing in (\ref{eq.form-faible}) (see Remark~\ref{rmk.init} above).

\section{Extending the boundary condition}\label{sec.bdry}

Let us start by denoting by $\calG$ the space of all locally integrable functions $f$ on $\hat{\Omega}$ having a weak gradient in $\hat{\Omega}$ and such that, for all $\delta>0$, there exists $r>n$ (depending on $\delta$) such that $|\nabla f|\in L^r(\hat{\Omega}_{\delta})$.
Now define a space $\calE$ by:
\begin{equation}\label{def.E}
\calE:=\{f\in \calG:\ f\in L^1(\mu_0)\mbox{ and }|\nabla f|w_0\in L^1(\Omega)\}.
\end{equation}
It will be shown in this section that, under the assumptions of Theorem~\ref{main}, \eqref{eq.g} can be extended to test functions in $\calE$.

\begin{Remark}\label{rmk.E}
Let us make immediately two straightforward observations.
\begin{itemize}
\item[(i)] The integrability condition on $\left\vert \nabla f\right\vert$ in each $\hat{\Omega}_{\delta}$ readily implies that any $f\in\calG$ is bounded and continuous on $\hat{\Omega}_\delta$ for all $\delta>0$; in particular $f$ has to be continuous on $\hat{\Omega}$. Moreover it is clear that $\calG$ contains the space of all locally Lipschitz functions in $\hat{\Omega}$ with bounded Lipschitz constant on $\hat{\Omega}$.
\item[(ii)] The space $\calE$ defined above obviously contains the space of all locally Lipschitz function $f$ on $\hat{\Omega}$ with bounded local Lipschitz constant and satisfying $f\in L^1(\mu_0)$.
\item[(iii)] Finally, observe that both $\calE$ and $\calG$ are vector spaces enjoying the property that for any $f\in \calE$ (resp. $f\in\calG$), one has $f_+,f_-, |f|\in \calE$ (resp. $f_+,f_-, |f|\in \calG$) and $\max(|\nabla f_+|,|\nabla f_-|, |\nabla|f||)\leq |\nabla f|$ a.e. in $\Omega$ (with respect to Lebesgue's measure).
\end{itemize}
\end{Remark}

We now turn to prove the following improvement of our main theorem.
\begin{theorem} \label{main2}
Assume that $\mu_0$ is a nontrivial finite (positive) Radon measure in $\Omega$ satisfying (\ref{dOmega}) and let $w_0$ be the weight defined by (\ref{wmu}), so that it satisfies properties (A) and (B) in Proposition~\ref{solutioninfty}. Given a (signed) Radon measure $\mu$ on $\Omega$ satisfying $\mu(\Omega)=0$ and $|\mu|\leq\kappa\mu_0$ for some $\kappa>0$, let also $u\in L^\infty_{1/w_0}$ be the solution of $\diver v=\mu$ constructed in Proposition~\ref{solutioninfty}, so that it satisfies $|u|\leq C\kappa |w_0|$ a.e.\ on $\Omega$, where $C>0$ is independent of $\kappa$, $\mu$ and $\mu_0$. We then have, for all $g\in\calE$:
\begin{equation}\label{eq.E}
\int_\Omega u\cdot\nabla g=-\int_\Omega g\,d\mu.
\end{equation}
\end{theorem}
\begin{proof}
Given $\epsilon>0$, define a signed Radon measure $\mu_\epsilon$ on $\Omega$ by $\mu_\epsilon(A):=\mu(A\cap\hat{\Omega}_\epsilon)$ for all $A\subset \Omega$ (that is to say that $\mu_\epsilon$ is the restriction of $\mu$ to $\Omega\cap\hat{\Omega}_\epsilon$). We have in particular $|\mu_\epsilon|\leq|\mu|\leq\kappa \mu_0$, so that, by Proposition~\ref{solutioninfty}, if 
$$
u_{\epsilon}(x):=\int_{\Omega} G(x,y)\chi_{\hat{\Omega}_{\epsilon}\cap \Omega}(y)d\mu(y),
$$
then $|u_\epsilon|\leq C\kappa w_0$ as well as:
$$
\int_\Omega u_\epsilon\cdot\nabla\varphi=-\int_\Omega \varphi\,d\mu_\epsilon,
$$
for all $\varphi\in\calD(\R^n)$.\par
\noindent We shall show in a moment that one has, for any $g\in\calE$:
\begin{equation}\label{eq.gprovi}
\int_\Omega u_\epsilon\cdot\nabla g=-\int_\Omega g\,d\mu_\epsilon.
\end{equation}
Let us first show how the latter equality will imply (\ref{eq.E}). To this purpose, fix $g\in\calE$ and observe, on one hand, that one has, for all $x\in\Omega$:
$$
|u(x)-u_\epsilon(x)|\leq \int_{\Omega} |G(x,y)||1-\chi_{\Omega\cap \hat{\Omega}_\epsilon}(y)|\,d|\mu|(y).
$$
Since $\lim_{\epsilon\rightarrow 0}\chi_{\Omega\cap\hat{\Omega}_\epsilon}(y)=1$, inequality (\ref{eq.19}) and the Lebesgue dominated convergence theorem ensure that $u_\epsilon$ converges a.e. to $u$ when $\epsilon\to 0$. Writing then, a.e.\ on $\Omega$:
$$
|u_\epsilon\cdot\nabla g|\leq C\kappa w_0 |\nabla g|\in L^1(\Omega),
$$
and using the Lebesgue dominated convergence theorem again, we see that:
$$
\int_\Omega u_\epsilon\cdot\nabla g\to \int_\Omega u\cdot\nabla g,
$$
as $\epsilon\to 0$.

\noindent On the other hand, observe using the Lebesgue dominated convergence theorem once more (recall that $g\in L^1(\mu)$ by definition of $\calE$) that one has:
$$
\left|\int_{\Omega} g\,d\mu-\int_{\Omega} g\,d\mu_\epsilon\right|\leq  \int_\Omega |g| |1-\chi_{\Omega_\epsilon\cap\Omega}|\,d|\mu|\to 0,
$$
as $\epsilon\to 0$.

\noindent Combining the last two facts with (\ref{eq.gprovi}) then yields (\ref{eq.E}).\par

We now turn to the proof of (\ref{eq.gprovi}). To that purpose, fix $g\in\calE$, let $(\rho_k)\subseteq\calD(\R^n)$ be an approximate identity satisfying $\supp \rho_k\subseteq B(0,2^{-k})$ for all $k$ and define $$\varphi_k:=\rho_k*(g\chi_k)\in\calD(\R^n),
$$
where $\chi_k:=\chi_{\hat{\Omega}_{2^{-k}}}$ and $g$ is extended by $0$ outside $\Omega$~---~this convolution being well defined on the whole space, smooth since $\rho_k$ is smooth, and having compact support since $\Omega$ is bounded. We hence have for each $k$, according to Proposition~\ref{solutioninfty} :
\begin{equation}\label{eq.sol2}
\int_\Omega u_\epsilon\cdot \nabla\varphi_k=-\int_\Omega \varphi_k\,d\mu_\epsilon.
\end{equation}
Since $g$ is continuous in $\hat{\Omega}$, it is clear, moreover, that $\varphi_k$ converges uniformly to $g$ on $\hat{\Omega}_\epsilon$. It hence follows that one has:
\begin{equation}\label{eq.sol22}
\lim_{k\to\infty}\int_\Omega \varphi_k\,d\mu_\epsilon =\lim_{k\to\infty}\int_{\Omega\cap\hat{\Omega}_\epsilon} \varphi_k\,d\mu=\int_{\Omega\cap\hat{\Omega}_\epsilon} g\,d\mu=\int_\Omega g\,d\mu_\epsilon.
\end{equation}

On the other hand, let $\delta>0$ be associated to $\epsilon$ according to property (d) in Lemma~\ref{paths}. We claim that $u_\epsilon=0$ outside $\Omega\cap\hat{\Omega}_{\frac 23\delta}$. Indeed, if $u_{\varepsilon}(x)\neq 0$ for some $x\in\Omega$, there exists $y\in\Omega\cap \hat{\Omega}_{\varepsilon}$ such that $G(x,y)\neq 0$. Therefore, there exists $t\in [0,1]$ such that $\left\vert x-\gamma(t,y)\right\vert\leq\rho(t,y)\leq \frac 15 \hat{d}(\gamma(t,y))$. This implies that
\begin{eqnarray*}
\hat{d}(x) & \geq & \hat{d}(\gamma(t,y))-\left\vert x-\gamma(t,y)\right\vert\\
& \geq & \hat{d}(\gamma(t,y))-\frac 15 \hat{d}(\gamma(t,y))\geq \frac 45\delta>\frac 23 \delta.
\end{eqnarray*}
This means that one has $x\in\hat{\Omega}_{\frac 23 \delta}$.

Observe now that if $x\in\hat{\Omega}_{\frac 23 \delta}$ and $k\in\N$ satisfying $2^{-k}<\frac 13 \delta$ are given, one gets:
$$
\hat{d}(y)\ge \hat{d}(x)-\left\vert x-y\right\vert\ge \frac 23 \delta-2^{-k}>\frac 13 \delta>2^{-k};
$$
it hence follows that one has $\varphi_k(x)=\rho_k*g(x)$ for all such $x$ and $k$. Since $g$ has a weak gradient in $\hat{\Omega}$, we also have, for the same $x$ and $k$:
\begin{equation}\label{eq.kgrand}
\nabla \varphi_k=\rho_k*\nabla g.
\end{equation}
%{[Remark: we could also argue that, coming back to the definition of weak gradient of $g$, one has for $x\in\hat{\Omega}_{\frac 45 \delta}$ and $2^{-k}<\frac 45 \delta$: \Bk
%\begin{multline*}
%\nabla\varphi_k\Rd(x)\Bk=(\nabla\rho_k)*(g\chi_k)\Rd(x)\Bk=-\int_{B(x,2^{-k})} \nabla[\rho_k(x-\cdot)] \chi_k g=-\int_{\hat{\Omega}} \nabla[\rho_k(x-\cdot)]  g\\=\int_{\hat{\Omega}}\rho_k(x-\cdot) \nabla g=\int_{B(x,2^{-k})} \rho_k(x-\cdot) \nabla g=\rho_k*\nabla g(x).
%\end{multline*}
%Up to you...]}
Using the latter facts, one computes:
$$
\left| \int_\Omega u_\epsilon\cdot \nabla\varphi_k-\int_\Omega u_\epsilon\cdot\nabla g\right|\leq C\kappa  \int_{\hat{\Omega}_{\frac 23\delta}} |\nabla\varphi_k-\nabla g|w_0,
$$
(recall that for the Lebesgue measure, it does not matter if one integrates on $\hat{\Omega}_{\frac 23 \delta}$ or $\Omega\cap\hat{\Omega}_{\frac 23 \delta}$).
Now since one has $g\in\calE$, there exists $r>n$ for which one has $\nabla g\in {L^r(\hat{\Omega}_{\frac 23 \delta})}$. Using H\"older's inequality and (\ref{eq.kgrand}), we have, for $1<r'<\frac{n}{n-1}$ satisfying $\frac 1r + \frac{1}{r'}=1$:
$$
\int_{\hat{\Omega}_{\frac 23\delta}} |\nabla\varphi_k-\nabla g|w_0\leq  \|w_0\|_{L^{r'}\left(\hat{\Omega}_{\frac 23\delta}\right)} \|\nabla g-\rho_k*\nabla g\|_{L^{r}\left(\hat{\Omega}_{\frac 23\delta}\right)}.
$$
Yet using Lemma~\ref{propI1mu0} and the fact that one has, for $x\in\hat{\Omega}_{\frac 23\delta}$:
$$
\omega(x)\hat{d}(x)^{1-n}\leq  \left(\frac 23\delta\right)^{1-n}\mu_0(\Omega),
$$
we see that $\|w_0\|_{L^{r'}(\hat{\Omega}_{\frac 23\delta})}<+\infty$; since the sequence $(\rho_k*\nabla g)$ converges in $L^{r}(\hat{\Omega}_{\frac 23\delta})$ to $\nabla g$, we hence see that one has:
$$
\lim_{k\to\infty} \int_\Omega u_\epsilon\cdot \nabla\varphi_k=\int_\Omega u_\epsilon\cdot\nabla g,
$$
which, combined to (\ref{eq.sol2}) and (\ref{eq.sol22}), finishes the proof of (\ref{eq.gprovi}).
\end{proof}

We now come to prove the equivalence of the solvability of \eqref{eq.pbl} and some versions of Poincar\'e inequalities.

\section{Equivalence between the solvability of (\ref{eq.pbl}) and some Poincar\'e inequalities} \label{Poinc}
This section is devoted to the equivalence between the solvability of (\ref{eq.pbl}) and some versions of Poincar\'e inequalities. Let $w\in L^1(\Omega)$ be a positive weight. We define spaces $\calE$ and $\calG$ as in the beginning of Section \ref{sec.bdry}, using the weight $w$ instead of $w_0$.
\begin{Definition}
\begin{enumerate}
\item Say that $(P_1)$ holds if there exists $C>0$ such that, for all $f\in \calE$,
\begin{equation} \label{poincare1} \tag{$P_1$}
\int_{\Omega} \left\vert f(x)-f_{\Omega}\right\vert d\mu_0\leq C \int_{\Omega} \left\vert \nabla f\right\vert w,
\end{equation}
where $f_{\Omega}:=\frac 1{\mu_0(\Omega)}\int_{\Omega} f(x)d\mu_0(x)$.
\item Say that $(P_1^{\ast})$ holds if there exists $C>0$ such that, for all $f\in {\calG}$ such that $E:=\{f=0\}$ verifies $\mu_0(E)>0$, one has,
\begin{equation} \label{poincare1weak} \tag{$P_1^{\ast}$}
\int_{\Omega} \left\vert f(x)\right\vert d\mu_0\leq C\left(1+\frac{\mu_0(\Omega)}{\mu_0(E)}\right) \int_{\Omega} \left\vert \nabla f\right\vert w,
\end{equation}
where it is understood that the finiteness of the right-hand side of the inequality implies that $f\in L^1(\mu_0)$. 
\end{enumerate}
\end{Definition}
We first observe that, given a weight $w$,  \eqref{poincare1} and \eqref{poincare1weak} are equivalent:
\begin{Proposition}\label{prop.poincare}
Let $w\in L^1(\Omega)$ be a positive weight in $\Omega$. Then:
\begin{enumerate}
\item \eqref{poincare1} and \eqref{poincare1weak} are equivalent,
\item if \eqref{poincare1} or \eqref{poincare1weak} holds, then for all $f\in \calG$ with $\left\vert \nabla f\right\vert w\in L^1(\Omega)$, one has $f\in L^1(\mu_0)$.
\end{enumerate}
\end{Proposition}
\begin{proof}
Assume first that \eqref{poincare1} holds. We first check \eqref{poincare1weak} for functions $f\in \calE$ such that the set $E:=\{f=0\}$ verifies $\mu_0(E)>0$. Note that in this case, one gets:
$$
\left\vert f_{\Omega}\right\vert=\frac 1{\mu_0(E)}\int_E \left\vert f-f_{\Omega}\right\vert\, d\mu_0\leq \frac 1{\mu_0(E)}\int_{\Omega} \left\vert f-f_{\Omega}\right\vert\, d\mu_0,
$$
which entails
\begin{equation} \label{poincare1zero}
\int_{\Omega} \left\vert f\right\vert d\mu_0\leq \int_{\Omega} \left\vert f-f_{\Omega}\right\vert d\mu_0+\left\vert f_{\Omega}\right\vert \mu_0(\Omega)\lesssim \left(1+\frac{\mu_0(\Omega)}{\mu_0(E)}\right)\int_{\Omega} \left\vert \nabla f\right\vert w.
\end{equation}
Fix now $f\in \calG$ with $\left\vert \nabla f\right\vert w\in L^1(\Omega)$, let $E:=\{f=0\}$ and assume one has $\mu_0(E)>0$. For all $N\geq 1$, define $f_N=\max\left(-N,\min(f,N)\right)$, which still belongs to $\calG$ with $\left\vert \nabla f_N\right\vert\leq \left\vert \nabla f\right\vert$ almost everywhere in $\Omega$ (for the Lebesgue measure); we hence get $f_N\in\calE$. Since $\mu_0\left(\left\{f_N=0\right\}\right)\geq \mu_0(E)>0$, \eqref{poincare1weak} applied to $f_N\in\calE$ shows that
$$
\int_{\Omega} \left\vert f_N(x)\right\vert d\mu_0\lesssim \left(1+\frac{\mu_0(\Omega)}{\mu_0(E)}\right) \int_{\Omega} \left\vert \nabla f_N\right\vert w\leq \left(1+\frac{\mu_0(\Omega)}{\mu_0(E)}\right)\int_{\Omega} \left\vert \nabla f\right\vert w,
$$
and since $f_N(x)\rightarrow f(x)$ for all $x\in \Omega$, the Fatou lemma proves that $f\in L^1(\mu_0)$, and \eqref{poincare1weak} holds.\par
\noindent Assume now that \eqref{poincare1weak} holds for all $f\in \calG$ with $\mu_0\left(\left\{f=0\right\}\right)>0$. That \eqref{poincare1} holds for all functions $f\in \calE$ can be proved as in \cite[Section 3.2]{DMRT}. \par
\noindent Assume finally that \eqref{poincare1weak} holds, and let $f\in \calG$ be nonnegative with $\left\vert \nabla f\right\vert w\in L^1(\Omega)$. Since $f$ is continuous in $\hat{\Omega}$, there exists $t_0\geq 0$ such that $\mu_0\left(\left\{f\leq t_0\right\}\right)>0$. Define now $\tilde{f}=(f-t_0)_+$. It is plain to see that $f\in L^1(\mu_0)$ if and only if $\tilde{f}\in L^1(\mu_0)$. Since $\tilde{f}\in \calG$, $\left\vert \nabla\tilde{f}\right\vert\leq \left\vert \nabla f\right\vert$ and $\mu_0\left(\left\{\tilde{f}=0\right\}\right)>0$, \eqref{poincare1weak} applied to $\tilde{f}$ shows that $\tilde{f}\in L^1(\mu_0)$, so that the same is true for $f$. In the general case, apply this conclusion to $f^+$ and $f^-$. 
\end{proof}
\begin{Remark}\label{rmq.lip-poinc}
It follows from the preceding proof that having inequality \eqref{poincare1} for all locally Lipschitz functions on $\hat{\Omega}$ with bounded Lipschitz constants in $\Omega$ and belonging to $L^1(\mu_0)$, is equivalent to \eqref{poincare1weak} holding for locally Lipschitz functions in $\hat{\Omega}$ whose local Lipschitz constant is bounded in $\hat{\Omega}$.\end{Remark}

The following statement precises somewhat the statement of Theorem~\ref{equivpoincare} given in the introduction. We keep the notations of the previous section.
\begin{theorem} \label{variousequiv}
Let $\Omega$ and $\mu_0$ be as in the statement of Theorem \ref{main}, and define $d_{\hat{\Omega}}$ as before. The following conditions are equivalent:
\begin{enumerate}
\item[(a)] $d_{\hat{\Omega}}\in L^1(\mu_0)$,
\item[(b)] there exists a weight $w\in L^1(\Omega)$, $w>0$ a.e.\ satisfying the conclusions of Theorem~\ref{main},
\item[(c)] there exists a weight $w\in L^1(\Omega)$, $w>0$ a.e.\ yielding either \eqref{poincare1} or \eqref{poincare1weak}.
\end{enumerate}
\end{theorem}
\begin{Remark}\label{rmk.poinc-2} As the proof (combined to Remark~\ref{rmq.lip-poinc}) will show, all these statements are also equivalent to the following ones:
\begin{itemize}
\item[(c')] there exists a weight $w\in L^1(\Omega)$, $w>0$ a.e.\ yielding \eqref{poincare1} for all bounded locally Lipschitz functions in $\hat{\Omega}$ whose local Lipschitz constant is bounded in $\hat{\Omega}$;
\item[(c'')] there exists a weight $w\in L^1(\Omega)$, $w>0$ a.e.\ yielding \eqref{poincare1weak} for all locally Lipschitz functions in $\hat{\Omega}$ whose local Lipschitz constant is bounded in $\hat{\Omega}$.
\end{itemize}
\end{Remark}
\begin{proof}
That $(a)$ implies $(b)$ was established in Theorem \ref{main}, since one can take $w=w_0$ where $w_0$ is defined in (\ref{wmu}). Assume now that $(b)$ holds and pick up $g\in L^{\infty}(\Omega,\mu_0)$ with $\left\Vert g\right\Vert_{\infty}\leq 1$. By $(b)$ and Theorem \ref{main2}, there exists a vector-valued  function $u$ in $\Omega$ satisfying the following conditions:
\begin{itemize}
\item[(i)] $\int_{\Omega} u\cdot \nabla h=-\int_{\Omega} (g-g_{\Omega})h\,d\mu_0$ for all $h\in\calE$;
\item[(ii)] $\left\Vert \frac u{w}\right\Vert_{\infty} \lesssim 1$.
\end{itemize}
It follows that for any $f\in\calE$, we have:
$$
\left\vert \int_{\Omega} (f-f_{\Omega})gd\mu_0 \Bk\right\vert=\left\vert \int_{\Omega} (f-f_{\Omega})(g-g_{\Omega})d\mu_0\right\vert =\left\vert \int_{\Omega} u\cdot \nabla f\right\vert 
\lesssim \int_{\Omega} \left\vert \nabla f\right\vert \,w,
$$
which yields \eqref{poincare1}, and hence also \eqref{poincare1weak} by Proposition~\ref{prop.poincare}.

\noindent Assume now that (c) is fulfilled. Since $\mu_0(\Omega)>0$, there exist $y_0\in \Omega$ and $r_0>0$ such that $B(y_0,r_0)\subset \Omega$ and $\mu_0(B(y_0,r_0))>0$. Denoting by ${d}_{\hat{\Omega}}'$ the geodesic distance to $y_0$ in $\hat{\Omega}$, we observe that $d_{\hat{\Omega}}'$ is locally Lipschitz on $\hat{\Omega}$ with local Lipschitz constant less than $1$, meaning in particular that one has $d_{\hat{\Omega}}'\in\calG$. As a consequence of ($P_1^*$) applied to $f:=(d_{\hat{\Omega}}'-r_0)_+$, \Bk we then get:
$$
\int_{\Omega} f\,d\mu_0\lesssim \int_{\Omega}\left\vert \nabla f\right\vert w\leq \int_{\Omega} w<+\infty,
$$
which yields the integrability of $d^{\prime}_{\hat{\Omega}}$ with respect to $\mu_0$, hence \Bk (a) since condition \eqref{dOmega} is independent of the choice of $x_0$. \end{proof}
\begin{Remark}
Observe that, as indicated in Remark~\ref{rmk.poinc-2}, the proof of the fact that (c) implies (a) has only used \eqref{poincare1weak} for the function $(d'_{\hat{\Omega}}-r_0)_+$, which is locally Lipschitz in $\hat{\Omega}$ and has a local Lipschitz constant bounded by $1$ on $\hat{\Omega}$.
\end{Remark}

\begin{Remark}
Assume that the weight $w\in L^1(\Omega)$, $w>0$ a.e.\ yields a Poincar\'e inequality \eqref{poincare1}. It then follows from the previous proposition that one has $d_{\hat{\Omega}}\in L^1(\Omega)$, and hence that there exists \emph{a (perhaps different)} weight $\tilde{w}\in L^1(\Omega)$, $\tilde{w}>0$ a.e.\ (one can take $\tilde{w}=w_0$ as in \eqref{wmu}) yielding the solvability, for any (signed) Radon measure $\mu$ in $\Omega$ satisfying $\mu(\Omega)=0$ and $|\mu|\leq \kappa\mu_0$, of problem (\ref{eq.pbl}) by some vector field $u$ satisfying $\|\frac u{\tilde{w}} \|_\infty\leq C\kappa$. As it follows from the following abstract reasoning, one can in fact take $\tilde{w}=w$.

Suppose indeed that $w\in L^1(\Omega)$, $w>0$ a.e.\ yields \eqref{poincare1}. Fix $\kappa>0$ and $\mu$ be a (signed) Radon measure in $\Omega$ satisfying $\mu(\Omega)=0$ and $|\mu|\leq \kappa \mu_0$.

Introduce the spaces $L^1_w(\Omega,\R^n)$, consisting of all measurable vector fields $u$ satisfying $|u|w\in L^1(\Omega)$ (endowed with $\|u\|_{L^1_w}:=\|uw\|_1$) and $L^\infty_{1/w}(\Omega,\R^n)$, consisting of all measurable vector fields $u$ satisfying $\frac{|u|}{w}\in L^\infty(\Omega)$ (endowed with $\|u\|_{L^\infty_{1/w}(\Omega)}:=\left\|\frac{u}{w}\right\|_\infty$).
We also introduce the auxiliary space:
$$
\calF:=\left\{ v\in L^1_{w}(\Omega,\R^n): \text{there exists }g\in\calE\text{ with }v=\nabla g \text{ a.e.\ in }\Omega\right\},
$$
which is a subspace of $L^1_{w}(\Omega,\R^n)$. Define for all  $v\in\calF$:
$$
T(v):=-\int_\Omega g\,d\mu,
$$
if $v=\nabla g$ a.e.\ on $\Omega$, with $g\in\calE$ and observe that this is unambiguous due to the fact that if $h\in \calE$ verifies $\nabla h=0$ a.e., then
it is constant on $\Omega$, so that one has $\int_\Omega h\,d\mu=0$ (recall that $\mu(\Omega)=0$).

Using \eqref{poincare1}, we compute for $v\in \calF$ and  $g\in\calE$ satisfying $v=\nabla g$ a.e.\ on $\Omega$:
$$
|T(v)|=\left|\int_{\Omega} (g-g_\Omega)\,d\mu\right|\leq \kappa\int_\Omega |g-g_\Omega|\,d\mu_0\lesssim \kappa\int_\Omega |\nabla g| w=\kappa\|v\|_{L^1_{w}(\Omega)}.
$$
Hence by the Hahn-Banach theorem $T$ extends to a bounded linear operator on $L^1_{w}(\Omega,\R^n)$. There thus exists $u\in L^\infty_{1/w}(\Omega,\R^n)$ verifying $\|u\|_{L^\infty_{1/w}(\Omega)}=\|T\|\lesssim \kappa$ such that one has $T(v)=\int_\Omega u\cdot v$ for all $v\in L^1_{w}(\Omega,\R^n)$. This implies, for any $g\in\calE$:
$$
\int_\Omega u\cdot \nabla g= T(\nabla g)=-\int_\Omega g\,d\mu,
$$
and we hence see that the weight $w$ allows the existence of a solution $u\in L^\infty_{1/w}(\Omega)$ to the equation $\diver v=\mu$ satisfying the required estimate.
\end{Remark}

\end{document}